\documentclass[11pt]{amsart}
\usepackage [T1]{fontenc}
\usepackage [latin1]{inputenc}
\usepackage{amssymb,amsmath,amsthm}
\usepackage{verbatim}
\usepackage{graphicx}
\usepackage{epsfig}

\usepackage{color, soul}
\usepackage[all]{xy}

\definecolor{dgreen}{cmyk}{1,.5,1,.2}

\newtheorem{theorem}{Theorem}[section]
\newtheorem{definition}[theorem]{Definition}

\newtheorem{example}[theorem]{Example}

\newtheorem{proposition}[theorem]{Proposition}
\newtheorem{corollary}[theorem]{Corollary}
\newtheorem{remark}[theorem]{Remark}

\def\s{\backslash}
\def\P{\mathcal{P}^n}
\def\Q{\mathcal{Q}}
\def\U{\mathfrak{A}}
\begin{document}

\title[The Symmetric Radon Nikodým property for tensor norms]{The Symmetric Radon Nikodým property for tensor norms.}

\author{Daniel Carando}

\author{Daniel Galicer}

\thanks{The first author was partially supported by ANPCyT PICT 05 17-33042, UBACyT Grant X038 and ANPCyT PICT 06 00587. The second author was partially supported by ANPCyT PICT 05 17-33042, UBACyT Grant X863 and a Doctoral fellowship from CONICET}

\address{Departamento de Matem\'{a}tica - Pab I,
Facultad de Cs. Exactas y Naturales, Universidad de Buenos Aires,
(1428) Buenos Aires, Argentina, and CONICET.} \email{dcarando@dm.uba.ar}
\email{dgalicer@dm.uba.ar}

\maketitle

\begin{abstract} We introduce the symmetric-Radon-Nikodým property (sRN property) for
finitely generated s-tensor norms $\beta$ of order $n$
and prove a Lewis type theorem for s-tensor norms with this
property. As a consequence, if $\beta$ is a projective s-tensor norm with the sRN property, then
for every Asplund space $E$, the canonical map
$\widetilde{\otimes}_{ \beta }^{n,s} E' \rightarrow
\Big(\widetilde{\otimes}_{ \beta' }^{n,s} E \Big)'$ is a
metric surjection.  This can be rephrased as the isometric isomorphism $\mathcal{ Q }^{min}(E) = \mathcal{Q}(E)$ for certain polynomial ideal $\Q$.
We also relate the sRN property of an s-tensor norm with the Asplund or Radon-Nikod\'{y}m properties of different tensor products. Similar results for full tensor products are also given.
As an application, results concerning the ideal of $n$-homogeneous extendible polynomials are obtained, as well as a new proof of the well known isometric isomorphism between nuclear and integral polynomials on Asplund spaces.
\end{abstract}

\section*{Introduction}

A result of Boyd and Ryan \cite{BoydRyan01} and also of Carando and Dimant \cite{CarDim00}
implies  that, for an Asplund space $E$, the space $\P_I(E)$ of
integral polynomials is isometric to the space $\P_N(E)$ of
nuclear polynomials (the isomorphism between these spaces was previoulsy obtained by Alencar in \cite{Alencar85,Alencar85(reflexivity)}).
In other words, if $E$ is Asplund, the space of integral
polynomials on $E$ coincides isometrically with its minimal hull $(\P_I)^{min}(E)=\P_N(E)$. This fact was used, for example, in \cite{BoydRyan01,BoyLasXX,Din03} to study geometric properties of spaces of polynomials and tensor products (e.g., extreme and exposed points of their unit balls), and in \cite{BoyLas05,BoyLas06} to characterize isometries between spaces of polynomials and centralizers of symmetric tensor products. When the above mentioned isometry is stated as the isometric coincidence between a maximal ideal and its minimal hull, it resembles the Lewis theorem for operator ideals and (2-fold) tensor norms (see \cite{Lewis} and also \cite[Section 33.3]{DefFlo93}). The Radon-Nikod\'{y}m property for tensor norms is a key ingredient for Lewis theorem.

The aim of this article is to find conditions under which the
equality $\Q(E)=\Q^{min}(E)$ holds isometrically for a maximal polynomial ideal
$\Q$.
In terms of symmetric tensor products,
we want conditions on an s-tensor norms ensuring the isometry
$\widetilde{\otimes}_{\beta}^{n,s} E' =
\big(\widetilde{\otimes}_{\beta'}^{n,s}E\big)'$.
To this end, we introduce the symmetric Radon-Nikod\'{y}m property for s-tensor norms and we show our main result, a Lewis-type Theorem (Theorem~\ref{Lewis theorem for polynomials}): if
an s-tensor norm has the symmetric Radon-Nikodým property (sRN
property),  we have that the canonical map $
\widetilde{\otimes}_{\s \beta /}^{n,s} E'  \rightarrow
\Big(\widetilde{\otimes}_{/ \beta' \s}^{n,s} E \Big)'$ is a
metric surjection for every Asplund space $E$ (see the notation below). As a consequence,
if $\mathcal{Q}$ is the maximal ideal (of n-homogeneous
polynomials) associated with a projective s-tensor norm $\beta$ with the sRN property, then
$\mathcal{ Q }^{min}(E) = \mathcal{Q}(E)$ isometrically.

As an application of this result, we reprove the isometric isomorphism between integral and nuclear polynomials on Asplund spaces (note that the result proved in \cite{BoydRyan01,CarDim00} is stronger). We also show that the ideal of extendible polynomials coincide with its minimal hull for Asplund spaces, and obtain as a corollary that the space of extendible polynomials on $E$ has a monomial basis whenever $E'$ has a basis.

We present examples of s-tensor norms associated to well known polynomial ideals which have the sRN property. We also relate the sRN property of an s-tensor norm with the Asplund property. More precisely, we show that, for $\beta$ is projective with the sRN, then $\beta'$ preserves the Asplund property, in the sense that $\widetilde{\otimes}_{ \beta' }^{n,s} E $ is Asplund whenever $E$ is. As an application, we show that the space of extendible polynomials on $E$ has the Radon-Nikod\'{y}m property if and only if $E$ is Asplund.  One might be tempted to infer that a projective  $\beta$ with the sRN property preserves the Radon-Nikod\'{y}m property, but this is not the case, as can be concluded from a result by Bourgain and Pisier~\cite{BouPis83}. However, we show that this is true with additional assumptions on the space $E$.

In order to prove our main theorem, we must show an analogous result for full tensor norms, which we feel can be of independent interest. It should be noted that, although we somehow follow some ideas of Lewis Theorem's proof in \cite[Section 33.3]{DefFlo93}, that proof is based on some factorizations of linear operators and not on properties of bilinear forms. Therefore, the weaker nature of the symmetric Radon-Nikod\'{y}m property introduced in this work together with our multilinear/polynomial framework makes our proof more complicated. As a consequence, we decided to postpone the proof of the main result to  Section~\ref{proof of big heorem}. The article is then organized as follows:
In Section \ref{preliminares} we recall same basic definitions
and facts about the theory of full and symmetric tensor norms and set some notation. In
Section~\ref{Seccion-SRN} we define the symmetric Radon-Nikod\'{y}m property, state our main theorem and prove the related results described above. We also exhibit some examples of tensor norms having the sRN property.
In Section~\ref{Seccion-full simetricos} we consider the sRN property in full tensor products and show the Lewis-type result for this spaces (Theorem~\ref{Lewis theorem}). In Section~\ref{proof of big heorem} we give the proof of Theorem~\ref{Lewis theorem for polynomials} and conclude the article with some questions.

We refer to \cite{DefFlo93} for the theory of tensor norms and operator ideals, and to \cite{Flo97,Flo01,Flo02(On-ideals),Flo02(Ultrastability)} for symmetric tensor products and polynomial ideals.

\section{Preliminaries}\label{preliminares}

For a natural number $n$, a full tensor norm $\alpha$ of order $n$ assigns to every $n$-tuple of Banach spaces $(E_1, \dots, E_n)$ a norm $\alpha \big(\; . \; ; \otimes_{i=1}^n E_i \big)$ on the $n$-fold (full) tensor product $\otimes_{i=1}^n E_i$ such that

\begin{enumerate}
\item $\varepsilon \leq \alpha \leq \pi$ on $\otimes_{i=1}^n E_i$.
\item $\| \otimes_{i=1}^n T_i :  \big( \otimes_{i=1}^n E_i, \alpha \big) \to \big( \otimes_{i=1}^n F_i, \alpha \big) \|  \leq \|T_1\| \dots \|T_n\|$ for each set of operator $T_i~\in~\mathcal{L}(E_i, F_i)$, $i=1, \dots, n$.
\end{enumerate}

We say that $\alpha$ is finitely generated if for all $E_i \in BAN$ (the class of all Banach spaces) and $z$ in $\otimes_{i=1}^n E_i$
$$ \alpha ( z, \otimes_{i=1}^n E_i ) : = \inf \{ \alpha (z, \otimes_{i=1}^n M_n ) : z \in \otimes_{i=1}^n M_i \},$$
the infimum being taken over all $n$-tuples $M_1, \dots, M_n$ of finite dimensional subspaces of $E_1, \dots, E_n$ respectively whose tensor product contains $z$.

We often call these tensor norms ``full tensor norms'', in the sense that they are defined on the full tensor product, to distinguish them from the s-tensor norms, that are defined on symmetric tensor products.

We say that  $\beta$ is a s-tensor norm  of order $n$ if $\beta$ assigns to each Banach space $E$ a norm $\beta \big(\; . \;; \otimes^{n,s} E \big)$ on the $n$-fold symmetric tensor product $\otimes^{n,s} E$ such that
\begin{enumerate}
\item $\varepsilon_s \leq \beta \leq \pi_s$ on $\otimes^{n,s} E$.
\item $\| \otimes^{n,s} T :  \big( \otimes^{n,s} E, \beta \big) \to \big( \otimes^{n,s} F, \beta \big) \| \leq \|T\|^n$ for each operator $T~\in~\mathcal{L}(E, F)$.
\end{enumerate}

$\beta$ is called finitely generated if for all $E \in  BAN$ and $z \in \otimes^{n,s} E$
$$ \beta (z, \otimes^{n,s}E) = \inf \{ \alpha(z, \otimes^{n,s}M) : M \in FIN(E), z \in \otimes^{n,s}M \}.$$

In this article we will only work with finitely generated tensor norms. Therefore, we will assume that all tensor norms are always finitely generated.

If $\alpha$ is a full tensor norm of order $n$, then the dual tensor norm $\alpha'$ is defined on $FIN$ (the class of finite dimensional Banach spaces) by

$$  \big( \otimes_{i=1}^n M_i, \alpha' \big) :\overset 1 = [\big( \otimes_{i=1}^n M_i', \alpha \big)]'$$
and on BAN by
$$ \alpha' ( z, \otimes_{i=1}^n E_i ) : = \inf \{ \alpha' (z, \otimes_{i=1}^n M_n ) : z \in \otimes_{i=1}^n M_i \},$$
the infimum being taken over all $n$-tuples $M_1, \dots, M_n$ of finite dimensional subspaces of $E_1, \dots, E_n$ respectively whose tensor product contains $z$.

Analogously, for $\beta$ an s-tensor norm of order $n$, its dual tensor norm $\beta'$ is defined on $FIN$ by

$$  \big( \otimes^{n,s} M, \beta' \big) :\overset 1 = [\big( \otimes^{n,s} M', \beta \big)]'$$
and extended to $BAN$ as before.

%

The projective and injective associates (or hulls) of $\alpha$ will be  denoted, by extrapolation of the 2-fold case, as $\s \alpha /$ and $/ \alpha \s$ respectively. The projective associate of $\alpha$ will be the (unique) smallest projective tensor norm greater than $\alpha$. Following \cite[Theorem 20.6.]{DefFlo93} we have:
$$ \big( \otimes_{i=1}^n \ell_1(E_i),  \alpha  \big) \overset 1 \twoheadrightarrow \big( \otimes_{i=1}^n  E_i,   \s \alpha /  \big).$$

The injective associate of $\alpha$ will be the (unique) greatest injective tensor norm smaller than $\alpha$.  As in \cite[Theorem 20.7.]{DefFlo93} we get,
$$ \big( \otimes_{i=1}^n E_i , / \alpha \s \big) \overset 1 \hookrightarrow  \big( \otimes_{i=1}^n \ell_{\infty}(B_{E_i'}),  \alpha  \big).$$

With this, an $n$-linear form $A$ belongs to $\big( \otimes_{i=1}^n  E_i,   \s \alpha /  \big)'$ if and only if $A \circ ( Q_{E_1}, \dots, Q_{E_n}) \in \big( \otimes_{i=1}^n  \ell_1(B_{E_i}),   \alpha   \big)'$ where $Q_{E_k} : \ell_1(B_{E_k}) \twoheadrightarrow E_k$ stands for the canonical quotient map.
Moreover, $$\|A\|_{\big( \otimes_{i=1}^n  E_i,   \s \alpha /  \big)'}= \|A \circ ( Q_{E_1}, \dots, Q_{E_n}) \|_{\big( \otimes_{i=1}^n  E_i,   \alpha  \big)'}.$$

On the other hand, an $n$-linear form $A$ is  $\big( \otimes_{i=1}^n  E_i,  /  \alpha \s  \big)'$ if it has an extension to $\ell_{\infty}(B_{E_1'})\times\dots\times \ell_{\infty}(B_{E_n'})$ that is $ \big( \otimes_{i=1}^n  \ell_{\infty}(B_{E_i'}),   \alpha   \big)'$. Moreover, the norm of $A$ in $\big( \otimes_{i=1}^n  E_i,   / \alpha \s  \big)'$ is the infimum of the norms in $\big( \otimes_{i=1}^n  \ell_{\infty}(B_{E_i'}),    \alpha   \big)'$ of all such extensions.

We will say that a tensor norm $\alpha$ is injective if $\alpha =
/ \alpha \s$. Equivalently, $\alpha$ is projective whenever $\alpha = \s \alpha /$.
%
%

Note that in our notation, the symbols ``$\s$'' and ``$/$'' by themselves lose their original meanings, as well as the left and right sides of $\alpha$.

The projective and injective associates for a s-tensor norm $\beta$ can be defined in a similar way:
$$ \big( \otimes^{n,s} \ell_1(E),  \beta  \big) \overset 1 \twoheadrightarrow \big( \otimes^{n,s}  E,   \s \beta /  \big).$$
$$ \big( \otimes_{i=1}^n E, / \beta \s \big) \overset 1 \hookrightarrow  \big( \otimes^{n,s} \ell_{\infty}(B_{E'}),  \beta \big).$$
The s-tensor norm $\beta$ will be called injective or projective if $\beta = / \beta \s$ or $\beta = \s \beta /$ respectively.


The description of the $n$-homogeneous polynomial $Q$ belonging to $\big( \otimes^{n,s}  E,   \s \beta /  \big)'$ or to $\big(\otimes^{n,s}~E,~/\beta\s\big)'$ is analogous to that for multilinear forms.

It is not hard to check, following the ideas of \cite[Proposition 20.10.]{DefFlo93}, the following duality relations for a full tensor norms $\alpha$ or an s-tensor norm $\beta$:
$$ (/ \alpha \s)' = \s \alpha ' /, \; \; \; (\s \alpha /)' = /\alpha '\s, \; \; \;   (/ \beta \s)' = \s \beta' /, \; \; \;   (\s \beta /)' = / \beta' \s.$$

If $\mathcal{U}$ is the maximal ideal of $n$-linear forms associated to the tensor norm $\alpha$ we will denote $\s \mathcal{U} /$ the maximal ideal associated to $\s \alpha /$ (i.e. $\s \mathcal{U} / (E_1, \dots, E_n)= \big( \otimes_{i=1}^n  E_i,   / \alpha \s  \big)'$).
Analogously, if $ \mathcal{Q}$ the maximal ideal of $n$-homogeneous polynomials associated to the s-tensor norm $\beta$ we will denote $\s \mathcal{Q} /$ the maximal ideal associated to $\s \beta /$.

\section{The symmetric Radon-Nikod\'{y}m property} \label{Seccion-SRN}

It is well know that the Radon-Nikodým property permitted to
understand the full duality of a tensor norm $\pi$ and
$\varepsilon$, describing conditions under which
$E'\widetilde{\otimes}_\pi
F'=(E\widetilde{\otimes}_{\varepsilon}F)'$ holds. Lewis in \cite{Lewis}
obtained many results of the form $E'\widetilde{\otimes}_\alpha
F'=(E\widetilde{\otimes}_{\alpha'}F)'$ or, in other words,
results about $\U^{min}(E,F')=\U(E,F')$ (if $\U$ is the maximal
operator ideal associated with $\alpha$).

For $\mathcal{Q}$ a maximal ideal of $n$-homogeneous polynomials, we want to find conditions under which the next equality holds:
\begin{equation}\label{igualdad}
\mathcal{Q}^{min}(E) = \mathcal{Q}(E).
\end{equation}

A related question is the following: If $\beta$ is the s-tensor norm of order $n$ associated to $\mathcal{Q}$, when does the natural map
\begin{equation}\label{aplicacionJ}
J: \widetilde{\otimes}^{n,s}_{\beta}E' \overset{1}{\twoheadrightarrow} \mathcal{Q}^{min}(E) \hookrightarrow \mathcal{Q}(E)\overset{1}{=} \big( \widetilde{\otimes}^{n,s}_{\beta'}E\big)',
\end{equation}
become a metric surjection? Note that, in this case, we get the equality (\ref{igualdad}).
To give an answer to this question we will need the next definition. In a sense, is symmetric version of the one which appears in \cite[33.2]{DefFlo93}.

\begin{definition}
A finitely generated s-tensor norm $\beta$ of order $n$ has the symmetric-Radon-Nikodým property (sRN property) if
$$  \widetilde{\otimes}_{\beta}^{n,s} \ell_1 \overset{1}{=} \Big(\widetilde{\otimes}_{\beta'}^{n,s} c_0 \Big)'.$$ Here equality means that canonical arrow $J: \widetilde{\otimes}_{\beta}^{n,s} \ell_1 \longrightarrow \Big(\widetilde{\otimes}_{\beta'}^{n,s} c_0 \Big)'$ (as in (\ref{aplicacionJ}) with $E=c_0$) is an isometric isomorphism.
\end{definition}

It is clear from the definition that a s-tensor norm $\beta$ has de sRN property if and only if its projective hull $\s \beta /$ does.

Since $\ell_1$ has the metric approximation property, by \cite[Corollary 5.2 and Proposition 7.5]{Flo01} we have that $J$ is always an isometry. Therefore, to prove that $\beta$ has the sRN property we only have to check that $J$ is surjective. Note that, for $\mathcal{Q}$ its the maximal $n$-homogeneous polynomial ideal associated to $\beta$, our previous definition is equivalent to
\begin{equation}
\label{sRNpolinomial}\mathcal{Q}^{min}(c_0) = \mathcal{Q}(c_0),
\end{equation}
and the isometry is automatic.
\bigskip

Our interest in the sRN property is the following Lewis-type
theorem:

\begin{theorem}\label{Lewis theorem for polynomials}
Let $\beta$ be an s-tensor norm with the sRN property and $E$ be an Asplund space.
Then
$$ \widetilde{\otimes}_{\s \beta /}^{n,s} E'  \twoheadrightarrow \Big(\widetilde{\otimes}_{/ \beta' \s}^{n,s} E \Big)'$$ is a metric surjection.

As a consequence, if $\Q$ is the maximal ideal of n-homogeneous polynomials associated with $\beta$, then
$$(\s\mathcal{ Q }/)^{min}(E) = \s\mathcal{Q}/(E) \; \mbox{isometrically}.$$
\end{theorem}

One may wonder if the projective hull of the tensor norm $\beta$ is really necessary in Theorem~\ref{Lewis theorem for polynomials}. Let us see that, in general, it cannot be avoided.
Take any injective s-tensor norm and let $\mathcal{Q}$ be the associated maximal polynomial ideal. If $T$ is dual of the original Tsirelson space (which is reflexive and therefore Asplund), then we can see that $\mathcal Q(T) \neq \mathcal Q^{min}(T)$. Indeed, if we write each $x\in T$ in terms of the canonical basis $x=\sum_j x_j e_j$, we can define a sequence of $n$-homogeneous polynomials  $P_m$ on $T$ as $P_m(x)=\sum_{j=1}^m x_j^n$. Since $\beta$ is injective, by \cite[Lemma 3.7.]{CarGal09}, we have that $\| P_m\|_{Q(\ell_2) }$ is uniformly bounded. Since $T$ does not contain $(\ell_2^m)_{m}$ nor $(\ell_\infty^m)_{m}$ uniformly complemented (see
\cite[pages 33 and 66]{Casazza-Shura(Tsirelson)}), we can conclude that $\mathcal Q(T)$ cannot be separable by  \cite[Proposition 4.9]{CarGal09}. As a consequence, $\mathcal Q(T)$ cannot coincide with $\mathcal Q^{min}(T)$.

In order to prove Theorem~\ref{Lewis theorem for polynomials}, an analogous result for full tensor products (and multilinear forms) will be necessary. As a consequence, we postpone the proof of Theorem~\ref{Lewis theorem for polynomials} to Section~\ref{proof of big heorem}.

\smallskip
Let us then present different  tensor norms with the sRN. We begin with two simple examples. The following identities are simple and well known:  $$ \Big(\widetilde{\otimes}_{\pi_s'}^{n,s} c_0 \Big)'= \Big(\widetilde{\otimes}_{\varepsilon_s}^{n,s} c_0 \Big)'= \widetilde{\otimes}_{\pi_s}^{n,s} \ell_1 $$ and
$$ \Big(\widetilde{\otimes}_{\varepsilon_s'}^{n,s} c_0 \Big)'= \Big(\widetilde{\otimes}_{\pi_s}^{n,s} c_0 \Big)'= \widetilde{\otimes}_{\varepsilon_s}^{n,s} \ell_1 $$ (they easily follow from the analogous identities for full tensor products, since the symmetrization operator is a continuous projection). Therefore, we have:

\begin{example}
The tensor norms $\pi_s$ and $\varepsilon_s $ have the sRN property.
\end{example}

In \cite{Alencar85,Alencar85(reflexivity)}, Alencar showed that if $E$ is Asplund, then integral and nuclear polynomials on $E$ coincide, with equivalent norms. Later, Boyd and Ryan \cite{BoydRyan01} and, independently, Carando and Dimant \cite{CarDim00}, showed that this coincidence is isometric (with a slightly more general assumption: that $\widetilde{\otimes}_{\varepsilon_s}^{n,s} E$ does not contain a copy of $\ell_1$). Note that the isometry between nuclear and integral polynomials on Asplund spaces is an immediate consequence of Theorem~\ref{Lewis theorem for polynomials} for $\beta=\pi_s$:

\begin{corollary}
If $E$ is Asplund, then $\P_I(E)=\P_N(E)$ isometrically.
\end{corollary}

If we apply Theorem~\ref{Lewis theorem for polynomials} and \cite[Corollary 5.2]{Flo01} to $\beta=\varepsilon_s$, we obtain $$ \P_e(E)=(\P_e)^{min}(E)= \widetilde{\otimes}_{\s \varepsilon_s /}^{n,s} E' \; \mbox{isometrically}$$ for $E'$ with the bounded approximation property. Combining this with the main result in \cite{GreRy05} we have:

\begin{corollary}
Let $E$ be a Banach space such that $E'$ has a basis. Then, the monomials associated to this basis is a Schauder basis of $\P_e(E)$.
\end{corollary}

\bigskip

We now give other examples of s-tensor norms associated to well know maximal polynomial ideals having the sRN property.
\medskip

\textbf{The ideal of  $r$-factorable polynomials:}
For $n \leq r \leq \infty$, a polynomial $P \in \mathcal{P}^n(E)$ is called \textbf{$r$-factorable} \cite{Flo02(On-ideals)} if there is a positive measure space $(\Omega, \mu)$, an operator $T \in \mathcal{L}\big(E, L_r(\mu)\big)$ and $Q \in \mathcal{P}^n \big( L_r(\mu) \big)$ with $P = Q \circ T$. The space of all such polynomials will be denoted by $\mathcal{L}_r^n(E)$.
With
$$ \|P\|_{\mathcal{L}_r^n(E)} = \inf \{ \|Q\| \|T\|^n \; : \; P: E \overset T \longrightarrow L_r(\mu) \overset Q \longrightarrow \mathbb{K}\}.$$

\begin{example}\label{ejemplo-factorables}
Let $\rho_n^r$ be the s-tensor norm associated to $\mathcal{L}_r^n$ ($r \geq n \geq 2$). Then, $\rho_n^r$ has the s-RN property.
\end{example}


\begin{proof}
We can suppose that $r < \infty$ since $\mathcal{L}_\infty^n(E)=\mathcal{P}_e^n(c_0)$ \cite[Proposition 3.4]{Flo02(On-ideals)}.
For $P \in \mathcal{L}_r^n(c_0)$ there is a measure space $(\Omega, \mu)$, an operator $T \in \mathcal{L}\big(c_0, L_r(\mu)\big)$ and a polynomial $Q \in \mathcal{P}^n \big( L_r(\mu) \big)$ with $P = Q \circ T$.
Since $L_r(\mu)$ is reflexive, as a direct consequence of the Schauder Theorem  and the Schur property of $\ell_1$ we have that  $T$ is approximable. On the other hand $Q$ is trivially in ${\mathcal{L}_r^n(L_r(\mu))}$. Hence $P$ belongs to $(\mathcal{L}_r^n)^{min}(c_0)$.
\end{proof}

\textbf{The ideal of  positively $r$-factorable polynomials:}
An $n$-homogeneous polynomial $Q: F \to \mathbb{K}$ on a Banach lattice $F$ is called positive, if $\check{Q}:F \to \mathbb{K}$ is positive, i.e., $\check{Q}(f_1, \dots, f_n)\geq 0$ for $f_1, \dots, f_n \geq 0$. For $n \leq r \leq \infty$, a polynomial $P \in \mathcal{P}^n(E)$ is called \textbf{postively $r$-factorable} \cite{Flo02(On-ideals)} if there is a positive measure space $(\Omega, \mu)$, an operator $T \in \mathcal{L}\big(E, L_r(\mu)\big)$ and $Q \in \mathcal{P}^n \big( L_r(\mu) \big)$ positive with $P = Q \circ T$. The space of all such polynomials will be denoted by $\mathcal{J}_r^n(E)$.
With
$$ \|P\|_{\mathcal{L}_r^n(E)} = \inf \{ \|Q\| \|T\|^n \; : \; P: E \overset T \longrightarrow L_r(\mu) \overset Q \longrightarrow \mathbb{K}\}.$$

Using the the ideas of the previous proof we have:

\begin{example}\label{ejemplo-positively-factorables}
Let $\delta_n^r$ be the s-tensor norm associated to $\mathcal{J}_r^n$ ($2 \leq n \leq r < \infty$). Then, $\delta_n^r$ has the s-RN property.
\end{example}

\textbf{The ideal of $r$-dominated  polynomials:}

For $x_1,\dots,x_m\in E$, we define
$$
w_{r} \big((x_{i})_{i=1}^m\big) = \sup_{x' \in B_{E'}} \left(
\sum_{i} |<x', x_i>|^{r} \right)^{1/r}.
$$
A polynomial $P \in P^n(E)$ is
\textbf{$r$-dominated} (for $r \geq n$) if there exists $C>0$ such that for every
finite sequence $(x^{i})_{i=1}^{m} \subset E$ the following holds
\[
\left( \sum_{i=1}^{m} |P (x_i)|^{\frac{r}{n}}
\right)^{\frac{n}{r}} \leq C w_{r}((x_i)_{i=1}^{m})^n.
\]
We will denote the space of all such polynomials by  $\mathcal{D}_{r}^n(E)$.
The least of such constants $C$ is called the
\emph{$r$-dominated norm} and denoted $\|P\|_{\mathcal{D}_r ^n(E)}$.

In \cite[Section 4]{CarDimSev07}, a $n$-fold full tensor norm $\alpha_{r'}^n$ was introduced, so that the ideal of dominated multilinear forms is dual to $\alpha_{r'}^n$. If we use the same notation for the analogous s-tensor norm, we have that $(\alpha_{r'}^n)'$ is the s-tensor norm associated to $\mathcal{D}_{r}$.

\begin{example}\label{ejemplo-dominated}
The s-tensor norm  $({\alpha_{r'}^n})'$ has the s-RN property.
\end{example}
\begin{proof}
By \cite{Sch91} we know that $\mathcal{D}_r^n  = \mathcal{P}^n \circ \Pi_r$, where $\Pi_r$ is the ideal of $r$-summing operators (see \cite{Flo01} for notation). Thus, for $P \in \mathcal{D}_r^n(c_0)$ we have the factorization $P = Q \circ T$ where $T: c_0 \longrightarrow G$ is an $r$-summing operator and $Q : G \longrightarrow \mathbb{K}$ an $n$-homogeneous continuous polynomial. We may assume without lost of generality that $G=F'$ for a Banach space $F$ (think on the Aron-Berner extension).
By \cite[Proposition 33.5]{DefFlo93} the tensor norm $(\alpha_{r',1})'$ has the Radon Nikodým property. Using this, and the identity $(\alpha^t)'=(\alpha')^t$ (which holds for every tensor norm of order two $\alpha$) we easily get:
$$
\begin{aligned}
\Pi_r(c_0,G) &= \Pi_r(c_0,F')=\big(c_0 \otimes_{\alpha_{1,r'}} F \big)'= \big(F \otimes_{\alpha_{r',1}} c_0 \big)' = \\
& = F' \otimes_{(\alpha_{r',1})'} \ell_1 = \ell_1 \otimes_{(\alpha_{1,r'})'}F'=\ell_1 \otimes_{(\alpha_{1,r'})'}G.
\end{aligned}
$$
Therefore, we have proved that $\Pi_r(c_0,G)=(\Pi_r)^{min}(c_0,G)$. Now is easy that $\mathcal{D}_r^n (c_0)=(\mathcal{D}_r^n)^{min}(c_0)$.
 \end{proof}

\bigskip

A natural and important question about a tensor norm is if it preserves some Banach space property. The following result shows that the symmetric Radon-Nikod\'{y}m is closely related to the preservation of the Asplund property under tensor products:

\begin{theorem} \label{RN para tensores sim}
Let $E$ be Banach space and $\beta$ a projective s-tensor norm with sRN property. The tensor product $\widetilde{\otimes}^{n,s}_{\beta'}E$ is Asplund if an only if $E$ is Asplund.
\end{theorem}

\begin{proof}
Necessity is clear. For the converse, let $S$ be a separable subspace of $ \widetilde{\otimes}^{n,s}_{  \beta' } E  $ and let us see that it has a separable dual. We can take $(z_k)_{k \in \mathbb{K}}$ a sequence of elementary tensors such that $S$ is contained in the closed subspace spanned by them. Each $z_k$ can be written as
$$ z_k = \sum_{j=1}^{r(k)}  \otimes^n  x_j^{k}.$$
Let $F = \overline{[ x_j^{k} : 1 \leq j \leq r(k), \; k \in \mathbb{N} ]}$. Since $\beta'$ is injective, we have the isometric inclusion $S \overset{1}\hookrightarrow \widetilde{\otimes}^{n,s}_{ \beta'} F $.
Now, $F'$ is separable (since $E$ is Asplund) and therefore, by Theorem~\ref{Lewis theorem for polynomials}, the map $$ \widetilde{\otimes}^{n,s}_{  \beta' } F'   \longrightarrow  \big( \widetilde{\otimes}^{n,s}_{  \beta' } F  \big)'$$ is surjective. So, $\big( \widetilde{\otimes}^{n,s}_{  \beta' } F  \big)'$ is a separable Banach space and hence is also $S'$ (since we have a surjective map $\big(  \widetilde{\otimes}^{n,s}_{  \beta' } F  \big)'  \twoheadrightarrow S'$).
\end{proof}

%

The following is an application of the previous theorem to $\beta = \s \varepsilon_s /$:
\begin{corollary}
For a Banach space $E$ and $n \in \mathbb{N}$, $\P_e(E)$ has the Radon-Nikodým property if and only if $E$ is Alplund.
\end{corollary}

Looking at Theorem~\ref{RN para tensores sim} a natural question arises: if $\beta$ is a projective s-tensor norm with the sRN property, does $\widetilde{\otimes}_{ \beta }^{n,s} E$ have the Radon-Nikodým property whenever $E$ has the Radon-Nikodým property?  Burgain and Pisier \cite[Corollary 2.4]{BouPis83} presented a Banach space $E$ with the Radon-Nikod\'{y}m property such that $E\otimes_\pi E$ contains $c_0$ and, consequently, does not have the Radon-Nikod\'{y}m property. This construction gives a negative answer to our question since the copy of $c_0$ in $E\otimes_\pi E$ is actually contained in the symmetric tensor product of $E$ and $\pi_s$ (which has the sRN property) is equivalent to the restriction of $\pi $ to the symmetric tensor product.

However, $\widetilde{\otimes}_{ \beta }^{n,s} E$ inherits the Radon-Nikod\'{y}m property of $E$ if, in addition, $E$ is a dual space with the approximation property (this should be compared to~\cite{DieUhl76}, where an analogous result for the 2-fold projective tensor norm $\pi$ is shown):

\begin{corollary} \label{Coro sim 1}
Let $\beta$ be a projective s-tensor norm with the sRN property and $E$ a dual Banach space with the approximation property. Then,  $\widetilde{\otimes}_{ \beta }^{n,s} E$ has the Radon-Nikodým property if and only if $E$ does.
\end{corollary}
\begin{proof}

Let $F$ be a predual of $E$ and suppose $E$ has the Radon-Nikod\'{y}m property. Since $F$ is Asplund, by Theorem~\ref{RN para tensores sim} so is $\widetilde{\otimes}^{n,s}_{\beta'}F$. On the other hand, by Theorem~\ref{Lewis theorem for polynomials} we have a metric surjection $ \widetilde{\otimes}_{ \beta }^{n,s} E  \twoheadrightarrow \Big(\widetilde{\otimes}_{ \beta' }^{n,s} F \Big)'$. Since $E=F'$ has the approximation property, this mapping is also injective: this can be seen, for example, a consequence of the injectivity of the following mapping \cite[Section 4.3]{Flo97}:
$$ \widetilde{\otimes}_{ \pi_s }^{n,s} E \rightarrow \widetilde{\otimes}_{ \beta }^{n,s} E  \twoheadrightarrow \Big(\widetilde{\otimes}_{ \beta' }^{n,s} F \Big)' \hookrightarrow \P(F)
$$
Therefore,  $ \widetilde{\otimes}_{ \beta }^{n,s} E$ is the dual of an Asplund Banach space and therefore has the Radon-Nikod\'{y}m property.

Since $E$ is complemented in $\widetilde{\otimes}_{ \beta }^{n,s} E$, the converse follows.
\end{proof}

Any Banach space $E$ with a boundedly complete Schauder basis $\{e_k\}_k$ is a dual space with the Radon-Nikod\'{y}m property and the approximation property. Indeed,  $E$ turns out to be the dual of the subspace $F$ of $E'$ spanned by the dual basic sequence $\{e'_k\}_k$ (which is, by the way, a shrinking basis of $F$). Then we have \begin{equation}\label{isometria} \widetilde{\otimes}_{ \beta }^{n,s} E  \overset 1 = \Big(\widetilde{\otimes}_{ \beta' }^{n,s} F \Big)'\end{equation} The monomials associated to $\{e_k\}_k$ and to $\{e'_k\}_k$ with the appropriate ordering (see \cite{GreRy05}) are Schauder basis of, respectively, $\widetilde{\otimes}_{ \beta }^{n,s} E$ and  $\widetilde{\otimes}_{ \beta' }^{n,s} F$. By the equality~\eqref{isometria}, monomials form a boundedly complete Schauder basis of  $\widetilde{\otimes}_{ \beta }^{n,s} E$  and a shrinking Schauder basis of $\widetilde{\otimes}_{ \beta' }^{n,s} F$.

On the other hand, if we start with a Banach space $E$ with a shrinking Schauder basis and take $F$ as its dual, we are in the analogous situation with the roles of $E$ and $F$ interchanged. So we have:

\begin{corollary} \label{Coro sim 2}
Let $\beta$ be a projective s-tensor norm with the sRN property.
\begin{itemize}
\item[a)] If $E$ has a boundedly complete Schauder basis, then so does $\widetilde{\otimes}_{ \beta }^{n,s} E$.
\item[b)]  If $E$ has a shrinking Schauder basis, then so does $\widetilde{\otimes}_{ \beta' }^{n,s} E$.
\end{itemize}
\end{corollary}

The corresponding statement for the 2-fold full tensor norm $\pi$ was shown by Holub in~\cite{Hol71}.

\section{The sRN property for full symmetric tensor norms}\label{Seccion-full simetricos}

In order to prove Theorem~\ref{Lewis theorem for polynomials} we must first show an analogous result for full tensor products (see Theorem~\ref{Lewis theorem} below). So let us first introduce the sRN property for full tensor products in the obvious way:

\begin{definition}
A finitely generated full tensor norm of order $n$ $\alpha$ has the symmetric Radon-Nikodým property (sRN property) if
$$ ( \widetilde{\otimes}_{i=1}^n \ell_1, \alpha ) = \Big(\widetilde{\otimes}_{i=1}^n c_0, \alpha' \Big)'.$$
\end{definition}

As in \cite[Lemma 33.3.]{DefFlo93} we have the following symmetric result for ideals of multilinears form.

\begin{proposition} Let $\alpha$ be a finitely generated full tensor norm of order $n$ with the sRN property.  Then,
$$ ( \widetilde{\otimes}_{i=1}^n \ell_1(J_i), \alpha ) = \Big(\widetilde{\otimes}_{i=1}^n c_0(J_i), \alpha' \Big)'$$
holds isometrically for all index sets $J_1, \dots, J_n$.
\end{proposition}

\begin{proof}
Fix $J_1, \dots, J_n$ index sets, and let us denote $\mathcal{U}(c_0(J_1), \dots , c_0(J_n))=\Big(\widetilde{\otimes}_{i=1}^n c_0(J_i), \alpha' \Big)'$. We must show $\mathcal{U}(c_0(J_1) , \dots , c_0(J_n))=\mathcal{U}^{min}(c_0(J_1) , \dots , c_0(J_n))$ with equal norms.
Take $T \in \mathcal{U}\big(c_0(J_1), \dots, c_0(J_n)\big)$.
If we denote $L = \{ (j_1, \dots, j_n) : T(e_{j_1}, \dots, e_{j_n}) \neq 0 \}$, let us prove that $L$ is countable.
If not, there exist $(j_1^k, \dots, j_n^k)_{k \in \mathbb{N}}$ different indexes such that
$$|T(e_{j_1^k}, \dots, e_{j_n^k})| > \varepsilon.$$
Passing to subsequences we may assume without lost of generality that $e_{j_i^k}$ are weakly convergent sequences and, moreover, $e_{j_1^k} \overset{w}{\longrightarrow} 0$.
This contradicts the Littlewood-Bogdanowicz-Pe{\l}czy{\'n}ski property of $c_0$ \cite{Bogdanowicz,Pel57}.

Let $\Omega_k: J_1 \times \cdots \times J_n \longrightarrow J_k$ given by $\Omega_k (j_1, \dots, j_n)= j_k$.
And let $L_k$ the set $\Omega_k(L) \subset J_k$.
Consider, $\xi_k$ the mapping $c_0(J_k) \to c_0(L_k)$ given by
$$ (a_j)_{j \in J_k} \mapsto (a_j)_{j \in L_k}.$$
And the inclusion $\imath_k: c_0(L_k) \to c_0(J_k)$ defined by
$$ (a_j)_{j \in L_k} \mapsto (b_j)_{j \in J_k},$$
where $b_j$ is $a_j$ if $j \in L_k$ and zero otherwise.
Then, we can factor

$$ \xymatrix{ c_0(J_1) \times \dots \times c_0(J_1) \ar[rr]^{\; \; \; \; \;\; \; \; \; \;\; \; \; \; \;T} \ar[d]_{\xi_1 \times \dots \times \xi_n}
& & \mathbb{K}  \\
c_0(L_1) \times \dots \times c_0(L_n) \ar[rru]^{\overline{T}}},$$
where $\overline{T}:= T \circ (\imath_1 \times \dots \times \imath_n)$.
Since $\alpha$ has the symmetric Radon Nikodým property we have that $\mathcal{U}(c_0(L_1), \dots , c_0(L_n))=\mathcal{U}^{min}(c_0(L_1) , \dots , c_0(L_n))$ with equal norms. Therefore $\overline{T}$ is in $\mathcal{U}^{min}(c_0(L_1), \dots , c_0(L_n))$ with
$$ \|\overline{T}\|_{\mathcal{U}^{min}} = \|\overline{T}\|_{\mathcal{U}} \leq \|T\|_{\mathcal{U}}.$$
Thus $\overline{T}$ belongs to $\mathcal{U}^{min}(c_0(L_1), \dots , c_0(L_n))$ which implies that $T$ also is $\mathcal{U}^{min}(c_0(J_1) , \dots , c_0(J_n)).$
Moreover,
$$\|T\|_{\mathcal{U}^{min}}\leq \|\overline{T}\|_{\mathcal{U}^{min}} \|\xi_1 \times \dots \times \xi_n \| \leq \|T\|_{\mathcal{U}}.$$
\end{proof}

Let $E_1, \dots, E_n$ be Banach spaces. For every $k=1 \dots n$ denote by $I_k: E_k \to \ell_{\infty}(B_{E_k'})$ the inclusion map. Also, let $EXT_k$ denote the canonical extension of a multilinear form to the bidual in the $k$-th coordinate (i.e. the multilinear version of the canonical extension $\varphi^{\wedge}$ and $^\wedge \varphi$ of a bilinear for $\varphi$, as in \cite[Section 1.9]{DefFlo93}). It is important to mention that, as in \cite[Section 6.7.]{DefFlo93}, we also have (with a similar proof) a multilinear version of the \emph{Extension Lemma}. In other words, extending a multilinear form to any bidual preserves the norm as a linear functional on the tensor product, for any finitely generated tensor norm.

For $\varphi : E_1 \times
\dots \times E_n \to \mathbb{K}$ we denote $ \varphi^n$ the
associated $(n-1)$-linear mapping $ \varphi^n: E_1 \times \dots  \times E_{n-1} \to E_n'$.
Now, if $T : E_n' \to F'$ is a linear operator, then the $(n-1)$-linear mapping $\rho: E_1 \times \dots  \times E_{n-1} \times F' \to \mathbb{K}$
given by $T \circ \varphi^{n}$ induces a $n$-linear form on $E_1 \times
\dots \times E_{n-1}\times F$.
It is not hard to check that
$$\rho (e_1, \dots, e_{n-1}, f)= (EXT_n)\varphi(e_1, \dots, e_{n-1}, T'J_F(f)),$$
where $J_F : F \to F''$ is the canonical inclusion map.

For every $k=1 \dots n$ we define an operator
$$\Psi_k:  \Big( (\widetilde\otimes_{j=1}^{k-1}E_j)  \widetilde{\otimes}  c_0(B_{E_k'}) \widetilde{\otimes} (\widetilde\otimes_{j=k+1}^{n}E_j), / \alpha' \s  \Big)' \rightarrow \Big( \widetilde{\otimes}_{i=1}^n E_i, / \alpha' \s  \Big)',$$
by the composition
$\big((\widetilde\otimes_{j=1}^{k-1}Id_{E_k})  \widetilde{\otimes}  I_k \widetilde{\otimes} (\widetilde\otimes_{j=k+1}^{n}Id_{E_k})\big)' \circ EXT_k$.

%
%
%

The following remark is easy to see:
\begin{remark} \label{diagrama conmuta}
Let $E_1, \dots, E_n$ be Banach spaces. For every $k$ the following diagram conmutes:

\xymatrix{ \big( (\widetilde\otimes_{j=1}^{k-1}E_j')  \widetilde{\otimes}  \ell_1(B_{E_k'}) \widetilde{\otimes} (\widetilde\otimes_{j=k+1}^{n}E_j'), \s \alpha / \big) \ar[r] \ar[d]^{(\widetilde\otimes_{j=1}^{k-1}Id_{E_j'})  \widetilde{\otimes}  Q_k \widetilde{\otimes} (\widetilde\otimes_{j=k+1}^{n}Id_{E_j'})}
& \Big( (\widetilde\otimes_{j=1}^{k-1}E_j)  \widetilde{\otimes}  c_0(B_{E_k'}) \widetilde{\otimes} (\widetilde\otimes_{j=k+1}^{n}E_j), / \alpha' \s  \Big)' \ar[d]^{\Psi_k}  \\
(E_1' \widetilde{\otimes} \dots \widetilde{\otimes} E_{k-1}' \widetilde{\otimes} E_k' \widetilde{\otimes} E_{k+1}' \widetilde{\otimes} \dots \widetilde{\otimes} E_n', \s \alpha / ) \ar[r]
& \Big( \widetilde{\otimes}_{i=1}^n E_i, / \alpha' \s  \Big)' ,}
where $Q_k: \ell_1(B_{E_k'}) \twoheadrightarrow E_k'$ is the quotient map.
\end{remark}

Now an important proposition:

\begin{proposition} \label{cociente}
Let $E_1, \dots, E_n$ be Banach spaces. If $E_k$ is Asplund then $\Psi_k$ is a metric surjection.
\end{proposition}

\begin{proof}
We will prove it assuming that $k=n$ (the other cases are completely analogous). Notice that $\Psi_n$ has norm less or equal to one (since $EXT_n$ is an isometry).

Fix $\varphi \in \Big(\widetilde{\otimes}_{i=1}^n E_i, / \alpha' \s  \Big)'$ and $\varepsilon > 0$ and let $\widetilde{\varphi} \in \Big( (\widetilde{\otimes}_{i=1}^{n-1} E_i) \widetilde{\otimes} \ell_\infty(B_{E_n'}), / \alpha' \s  \Big)'$ a Hahn-Banach extension of $\varphi$.
Since $E_n'$ has the Radon-Nikodým property, by the Lewis-Stegall theorem the adjoint of the canonical inclusion $I_n: E_n \to \ell_{\infty}(B_{E_n'})$ factors through $\ell_1(B_{E_n'})$ via

\begin{equation} \label{factorizacion}
\xymatrix{ {\ell_{\infty}(B_{E_n'})'}  \ar[rr]^{I_n'} \ar[rd]^{A} & & {E_n'} \\
& {\ell_1(B_{E_n'})} \ar@{->>}[ru]^{Q_n} & }
\end{equation}
where $Q_n$ is the canonical quotient map and $\|A\| \leq (1+\varepsilon)$.
Let $\rho: E_1 \times \dots \times E_{n-1} \times c_0(B_{E_n'}) \to \mathbb{K}$ given by the formula
$\rho (x_1, \dots, x_{n-1}, a)= (EXT_n)\widetilde{\varphi}(x_1, \dots, x_{n-1}, A'J_{c_0(B_{E_n'})}(a))$ ($\rho$ is the $n$-linear form on $E_1 \times \dots \times E_{n-1} \times c_0(B_{E_n'})$ associated to  $A \circ (\widetilde{\varphi})^n$). Using the ideal property and the fact that the extension to the bidual is an isometry $\rho \in  \Big( (\widetilde{\otimes}_{i=1}^{n-1} E_i ) \widetilde{\otimes} c_0(B_{E_n'}), / \alpha' \s  \Big)'$ and $\|\rho\| \leq \|\varphi\| (1+\varepsilon)$.

If we show that $\Psi_n (\rho)=\varphi$ we are done. It is an
easy exercise to prove that $I_n' (\widetilde{\varphi})^n = \varphi^n$.
It is also easy to see that $I_n(x) (a) = Q_n(a)(x)$ for $x \in
E_n$ and $a \in \ell_1(B_{E_n'})$.

Now, $\Psi_n (\rho)=( \widetilde{\otimes}_{i=1}^{n-1} Id_{E_i} \widetilde{\otimes} I_n )' \circ
(EXT_n) (\rho)$. Then,
\begin{align*}
\Psi_n (\rho)(x_1, \dots, x_n)  & = (I_nx_n) [\rho(x_1, \dots,x_{n-1},\cdot)] \\
& = (I_nx_n) A (\widetilde{\varphi})^n(x_1, \dots, x_{n-1}) \\
& = Q_n A (\widetilde{\varphi})^n(x_1, \dots, x_{n-1}) (x_n) \\
& = I_n' (\widetilde{\varphi})^n(x_1, \dots, x_{n-1}) (x_n)  \\
& = \varphi^n(x_1, \dots, x_{n-1}) (x_n) \\
& = \varphi(x_1, \dots, x_n).
\end{align*}
\end{proof}


\smallskip
The following result is the version of Theorem~\ref{Lewis theorem for polynomials} for full symmetric tensor products. It should be noted that it holds for tensor products of different spaces.

\begin{theorem}\label{Lewis theorem}
Let $\alpha$ be a tensor norm with the sRN property and $E_1, \dots, E_n$ be Asplund spaces.
Then
$$ (\widetilde{\otimes}_{i=1}^n E_i', \s \alpha / ) \twoheadrightarrow \Big(\widetilde{\otimes}_{i=1}^n E_i, / \alpha' \s \Big)'$$ is a metric surjection.

In particular if $(\mathfrak{A},\mathbf{A})$ is the maximal ideal (of multilinear forms) associated with $\alpha$, then
$$(\s\mathfrak{ A }/)^{min}(E_1, \dots, E_n) = \s\mathfrak{A}/(E_1, \dots, E_n).$$
\end{theorem}

\begin{proof}
Using Remark~\ref{diagrama conmuta} we know that the following diagram commutes in each square.

\begin{equation*}
\xymatrix{ \big( \widetilde{\otimes}_{i=1}^n \ell_1(B_{E_i'}), \s \alpha / \big) \ar[rr]^{R_0} \ar@{->>}[dd]^{\widetilde{\otimes}_{i=1}^{n-1} Id_{\ell_1(B_{E_i'})} \widetilde{\otimes} P_n }
&  & \Big( \widetilde{\otimes}_{i=1}^n c_0(B_{E_i'}), / \alpha' \s  \Big)' \ar[d]_{EXT_n} \\
& & \Big( (\widetilde{\otimes}_{i=1}^{n-1} c_0(B_{E_i'}) )  \widetilde{\otimes} \ell_{\infty}(B_{E_n'}), / \alpha' \s  \Big)' \ar@{->>}[d]_{\big( (\widetilde{\otimes}_{i=1}^{n-1} Id_{c_0(B_{E_i'})} ) \widetilde{\otimes} I_n \big)'} \\
( (\widetilde{\otimes}_{i=1}^{n-1} \ell_1(B_{E_i'}) ) \widetilde{\otimes} E_n', \s \alpha / ) \ar[rr]^{R_1} \ar@{->>}[dd]^{(\widetilde{\otimes}_{i=1}^{n-2} Id_{\ell_1(B_{E_i'})} ) \widetilde{\otimes} P_{n-1} \widetilde{\otimes} Id_{E_n'} }
& & \Big( (\widetilde{\otimes}_{i=1}^{n-1} c_0(B_{E_i'}) )  \widetilde{\otimes} E_n, / \alpha' \s  \Big)' \ar[d]_{EXT_{n-1}} \\
& & \Big( (\widetilde{\otimes}_{i=1}^{n-2} c_0(B_{E_i'}) ) \widetilde{\otimes} \ell_{\infty}(B_{E_{n-1}'}) \widetilde{\otimes} E_n, / \alpha' \s  \Big)' \ar@{->>}[d]_{\big( (\widetilde{\otimes}_{i=1}^{n-1} Id_{c_0(B_{E_i'})} )  \widetilde{\otimes} I_{n-1} \widetilde{\otimes} Id_{E_n'} \big)'} \\
((\widetilde{\otimes}_{i=1}^{n-2} \ell_1(B_{E_i'}) ) \widetilde{\otimes} E_{n-1}' \widetilde{\otimes} E_n', \s \alpha / ) \ar[rr]^{R_2} \ar@{->>}[d]
& & \Big( (\widetilde{\otimes}_{i=1}^{n-2} c_0(B_{E_i'}) ) \widetilde{\otimes} E_{n-1} \widetilde{\otimes} E_n, / \alpha' \s  \Big)' \ar[d] \\
& \dots & \\
\dots & \dots & \dots \\
\ar@{->>}[d] & \dots & \ar[d]\\
(\ell_1(B_{E_1'}) \widetilde{\otimes} (\widetilde{\otimes}_{i=2}^{n} E_i' ), \s \alpha / ) \ar[rr]^{R_{n-1}} \ar@{->>}[dd]^{P_1 \widetilde{\otimes} (\widetilde{\otimes}_{i=2}^{n} Id_{E_i'} )}
& & \Big( c_0(B_{E_1'}) \widetilde{\otimes} (\widetilde{\otimes}_{i=2}^{n} E_i), / \alpha' \s  \Big)'  \ar[d]_{EXT_1} \\
& & \Big( \ell_{\infty}(B_{E_1'}) \widetilde{\otimes} (\widetilde{\otimes}_{i=2}^{n} E_i), / \alpha' \s  \Big)' \ar@{->>}[d]_{\big(I_1 \widetilde{\otimes} (\widetilde{\otimes}_{i=2}^{n} Id_{E_i'}) \big)'} \\
(\widetilde{\otimes}_{i=1}^{n} E_i', \s \alpha / ) \ar[rr]^{R_n}
& & \Big( \widetilde{\otimes}_{i=1}^{n} E_i, / \alpha' \s  \Big)'  \\
}
\end{equation*}

Looking at the first commutative diagram and using that $R_0$ is a metric surjection (Proposition~\ref{cociente}) we get that $R_1$ is also a metric surjection.
The same argument can be applied to the second commutative diagram now that we know that $R_1$ is metric surjection. Thus, $R_2$ is also a metric surjection. Reasoning like this, it follows that $R_n : (\widetilde{\otimes}_{i=1}^n E_i', \s \alpha / ) \rightarrow \Big(\widetilde{\otimes}_{i=1}^n E_i, / \alpha' \s \Big)'$ is a metric surjection.
\end{proof}

We will call $\Psi : \big(  \widetilde{\otimes}_{i=1}^{n} c_0(B_{E_i}') \big)' \to \big( \widetilde{\otimes}_{i=1}^{n} E_i' \big)'$ the composition of the downward arrows in the right side of the last diagram. The following proposition shows how to describe the map $\Psi$ more easily (it will be useful to prove the polynomial version the of the last theorem).

\begin{proposition}\label{reescritura del mapa}
The arrow $\Psi :  \big( \widetilde{\otimes}_{i=1}^{n} c_0(B_{E_i}', /\alpha' \s)  \big)' \to \big( \widetilde{\otimes}_{i=1}^{n} E_i, /\alpha' \s \big)'$ is the composition map
$$ \big( \widetilde{\otimes}_{i=1}^{n} c_0(B_{E_i}'), /\alpha' \s \big)' \overset{EXT}\longrightarrow \big( \widetilde{\otimes}_{i=1}^{n} \ell_{\infty}(B_{E_i}'), /\alpha' \s \big)' \overset{(\widetilde{\otimes}_{i=1}^{n} I_i )'}\longrightarrow \big( \widetilde{\otimes}_{i=1}^{n} E_i, /\alpha' \s \big)',$$
where $EXT$ stands for the iterated extension to the bidual given by $(EXT_n) \circ \dots \circ (EXT_1)$ (we extend from the left to the right).
\end{proposition}

\begin{proof}
For the readers' sake we will give a proof for the case $n=2$. Let $\rho \in \big(c_0(B_{E_1'})  \widetilde{\otimes} c_0(B_{E_2'}) , /\alpha' \s \big)'$, then
\begin{align*}
\Psi(\rho) (e_1,e_2) & = ( id_{E_1} \widetilde{\otimes} I_2)'(EXT_2)(I_1 \widetilde{\otimes }Id_{c_0(B_{E_2'})})'(EXT_1)(\rho)(e_1,e_2) \\
& = (EXT_2)(I_1 \widetilde{\otimes }Id_{c_0(B_{E_2'})})'(EXT_1)(\rho)(e_1,I_2(e_2)) \\
& = I_2(e_2) \Big( (I_1 \widetilde{\otimes }Id_{c_0(B_{E_2'})})'(EXT_1)(\rho)(e_1,\cdot) \Big) \\
& = I_2(e_2) \Big( a \mapsto I_1(e_1) \rho (\cdot, a) \Big) \\
& = I_2(e_2) \Big( (EXT_1)(\rho)(I_1(e_1),\cdot) \Big) \\
& =(EXT) (\rho) (I_1(e_1),I_2(e_2)) \\
& = (I_1 \widetilde{\otimes} I_2)' (EXT) (\rho) (e_1,e_2).
\end{align*}
\end{proof}

Now, this proposition shows that the diagram

\begin{equation*}
\xymatrix{( \widetilde{\otimes}_{i=1}^{n} \ell_1(B_{E_i'}), \s \alpha / ) \ar@{->>}[r] \ar@{->>}[d]^{P}
& \Big( \widetilde{\otimes}_{i=1}^{n} c_0(B_{E_i'}), / \alpha' \s  \Big)' \ar@{->>}[d]^{\Psi}  \\
(\widetilde{\otimes}_{i=1}^{n} E_{i}', \s \alpha / ) \ar[r]
& \Big( \widetilde{\otimes}_{i=1}^{n} E_i, / \alpha' \s  \Big)' }
\end{equation*}
conmutes and, by the proof of the Theorem~\ref{Lewis theorem}, we have that, for $E_1, \dots, E_n$ Asplund spaces, the map $\Psi$ is a metric surjection.

The next remark will be very useful. It can be proved following carefully the proof of Proposition~\ref{cociente} and using Proposition~\ref{reescritura del mapa}.
\begin{remark}\label{comentario}
Let $E$ an Asplund space and $A: \ell_\infty(B_{E'})' \to \ell_1(B_{E'})$ be the operator obtained by the Lewis-Stegall Theorem  with $\|A\|\leq 1 + \varepsilon$ as in diagram (\ref{factorizacion}).
Given $\varphi \in \Big( \widetilde{\otimes}_{i=1}^{n}  E, / \alpha' \s  \Big)'$, if we take a Hahn-Banach extension $\widetilde{\varphi} \in \Big(\widetilde{\otimes}_{i=1}^{n}  \ell_\infty(B_{E'}), / \alpha' \s  \Big)'$, then the linear functional $\rho \in \Big( \widetilde{\otimes}_{i=1}^{n}  c_0(B_{E'}), / \alpha' \s  \Big)'$ given by
\begin{equation} \label{identidad piola}
\rho(a_1, \dots, a_n) := (EXT) (\widetilde{\varphi})(A'J(a_1), \dots, A'J(a_n)),
\end{equation}
satisfies $\Psi(\rho)=\varphi$ and $\|\rho\| \leq \| \varphi \| (1+\varepsilon)^n$.
\end{remark}

We end this section with the statement of the non-symmetric versions of Theorem \ref{RN para tensores sim}, Corollary \ref{Coro sim 1} and Corollary \ref{Coro sim 2}, which readily follow:
\begin{theorem} \label{RN para tensores}
Let $E_1, \dots, E_n$ be Banach spaces and $\alpha$ a full symmetric tensor norm with sRN. The tensor product $\big( E_1 \widetilde{\otimes} \dots \widetilde{\otimes}  E_n, / \alpha' \s  \big)$ is Asplund if an only if $E_i$ is Asplund for $i=1\dots n$ .
\end{theorem}

\begin{corollary} \label{Coro sim 1 full}
Let $\alpha$ be a projective full symmetric tensor norm with the sRN property and $E_1, \dots, E_n $ dual Banach spaces with the approximation property. Then, $\big( E_1 \widetilde{\otimes} \dots \widetilde{\otimes}  E_n,  \alpha \big)$ has the Radon-Nikodým property if and only if every $E_i$ does.
\end{corollary}

\begin{corollary} \label{Coro sim 2 full}
Let $\alpha$ be a projective full symmetric tensor norm with the sRN property and $E_1, \dots, E_n$ be Banch spaces.
\begin{itemize}
\item[a)] If every $E_i$ has a boundedly complete Schauder basis, then so does $\big( E_1 \widetilde{\otimes} \dots \widetilde{\otimes}  E_n,  \alpha  \big)$.
\item[b)]  If every $E_i$ has a shrinking Schauder basis, then so does $\big( E_1 \widetilde{\otimes} \dots \widetilde{\otimes}  E_n,  \alpha'  \big)$.
\end{itemize}
\end{corollary}

\section{The proof of Theorem~\ref{Lewis theorem for polynomials} and some questions}\label{proof of big heorem}
%

To prove Theorem~\ref{Lewis theorem} we used a multilinear version of the \emph{Extension Lemma} whose proof follows identical to the one in \cite[6.7.]{DefFlo93}. For polynomials a similar result is needed:

\begin{proposition}\cite[Corollary 3.4.]{CarGal09AB} \label{isom AB maximales} Let $\beta$ a finitely generated s-tensor norm. For each $P \in \big( \widetilde{\otimes}^{n,s}_\beta E \big)'$ its Aron-Berner extension $AB(P)$ of $P$ belongs to
$\big( \widetilde{\otimes}^{n,s}_\beta E'' \big)'$ and
$$ \|P\|_{\big( \widetilde{\otimes}^{n,s}_\beta E \big)'} = \|AB(P)\|_{\big( \widetilde{\otimes}^{n,s}_\beta E''
\big)'}.$$
\end{proposition}

This was obtained as a consequence of the isometry of the iterated extension to ultrapowers for maximal polynomial ideals. However, this can also be proved without the ultrapower techniques: just use local reflexivity instead of local determination of ultrapowers and proceed as in \cite{CarGal09AB}.

\begin{proof} (of Theorem~\ref{Lewis theorem for polynomials})

As in the multilinear case, the next diagram commutes:

\begin{equation*}
\xymatrix{ \widetilde{\otimes}_{\s \beta /}^{n,s} \ell_1(B_{E'}) \ar@{->>}[r] \ar@{->>}[d]^{\widetilde{\otimes}^{n,s}P}
& \Big( \widetilde{\otimes}_{/ \beta' \s}^{n,s} c_0(B_{E'}) \Big)' \ar[d]^{\Psi}  \\
\widetilde{\otimes}_{\s \beta /}^{n,s} E' \ar[r]
& \Big( \widetilde{\otimes}_{/ \beta' \s}^{n,s} E \Big)' },
\end{equation*}
where $\Psi$ is the composition map
$$ \Big( \widetilde{\otimes}_{/ \beta' \s}^{n,s} c_0(B_{E'}) \Big)' \overset{AB}\longrightarrow \Big( \widetilde{\otimes}_{/ \beta' \s}^{n,s} \ell_{\infty}(B_{E'}) \Big)' \overset{(\widetilde{\otimes}^{n,s}I)'}\longrightarrow \Big( \widetilde{\otimes}_{/ \beta' \s}^{n,s} E \Big)'.$$

Fix $P \in \Big( \widetilde{\otimes}_{/ \beta' \s}^{n,s} E \Big)'$. Denote by $\overline{P}~\in~\Big( \widetilde{\otimes}_{/ \beta' \s}^{n,s} \ell_{\infty}(B_{E'}) \Big)'$ a Hahn-Banach extension of $P$ and by $A$ an operator obtained from the Lewis-Stegall theorem such that $\|A\| \leq 1 + \varepsilon$ (see diagram (\ref{factorizacion})).  Since the Aron-Berner is an isometry for maximal ideals (Proposition~\ref{isom AB maximales}) we have, as in Remark~\ref{comentario}, that the linear functional $L \in \Big(\widetilde{\otimes}_{/ \beta' \s}^{n,s} c_0(B_{E'}) \Big)'$ given by $L(a):= (AB)(\overline{P})(A'J_{c_0(B_{E'})}a)$ satisfies that $\Psi(L)=P$ and $\|L\|_{\Big(\widetilde{\otimes}_{/ \beta' \s}^{n,s} c_0(B_{E'}) \Big)'}\leq \|P\|_{\Big( \widetilde{\otimes}_{/ \beta' \s}^{n,s} E \Big)'} (1+ \varepsilon)^n.$ Thus, $\Psi$ is a metric surjection and, by the diagram, we easily get that $\widetilde{\otimes}_{\s \beta /}^{n,s} E' \rightarrow  \Big( \widetilde{\otimes}_{/ \beta' \s}^{n,s} E \Big)'$ is also a metric surjection.
\end{proof}

\smallskip

We conclude the article with a couple of questions:

Since we do not know of any example of an s-tensor without the sRN property, we ask: \emph{Does every s-tensor norm have the sRN property?}

A more precise, but not necessarily easier, question is the following: \emph{Does $/ \pi_s \s$ have the sRN property?} In the case of a positive answer we would have that every natural s-tensor norm (see \cite{CarGal09Nat}) have this property.

\end{document}